\documentclass[12pt]{article}

\usepackage[utf8]{inputenc}

\usepackage{amsthm}
\usepackage{tensor}
\usepackage{mathtools}
\usepackage{amsfonts}
\usepackage{amsmath}
\usepackage{amssymb}
\usepackage{upgreek}
\usepackage{multirow}
\usepackage{mathtools}
\usepackage{enumerate}
\usepackage{listings}
\usepackage{authblk}

\usepackage{tikz}
\usepackage{tikz-cd}
\usetikzlibrary{matrix,arrows}
\newtheorem{theorem}{Theorem}

\newtheorem{definition}{Definition}
\newtheorem{proposition}{Proposition}
\newtheorem{characterisation}{Characterisation}

\newtheorem{corollary}{Corollary}

\newtheorem{example}{Example}

\begin{document}
\author{Matthew Terje Aadne}
\title{Kundt Structures}
\affil{Faculty of Science and Technology,\\
University of Stavanger,\\
4036 Stavanger, Norway}
\affil{matthew.t.aadne@uis.no}
\maketitle
\begin{abstract}
    In this paper we consider a new approach to studying Kundt spacetimes through $G$-structures. We define a Lie-group $GN$ such that the $GN$-structures satisfying an integrability condition and an existence criterion, which we call Kundt structures, have the property that each metric belonging to the Kundt structure is automatically a Kundt spacetime. We find that the Lie algebra of infinitesimal automorphisms of such structures is given by a Lie algebra of nil-Killing vector fields. Lastly we characterize all left invariant Kundt structures on homogeneous manifolds.
\end{abstract}
\section{Introduction}
Kundt spacetimes \cite{KundtSpacetimes,ExactSolutions,GenKundt}, have been seen to form an important class of spacetimes for the study of scalar curvature invariants on Lorentzian manifolds. In \cite{CharBySpi} and \cite{SCPI} the authors identified a large subclass, the degenerate Kundt spacetimes, for which there exists smooth deformations of the metrics leaving the scalar curvature invariants fixed while going outside the orbit of the metrics. Conversely, they also showed that in dimension four, any Lorentzian metric having such a deformation must belong to the degenerate Kundt class. They key feature which gives rise to degeneracies in the scalar polynomial curvature invariance is the presence of some null-direction, for which the curvature tensor and all its covariant derivatives are of type $II.$ An open question of current interest is whether all such metrics in arbitrary dimensions belong to the Kundt class.

Another interesting aspect is the rich connection between Kundt spacetimes and \emph{CSI} metrics,  defined as metrics for which all scalar polynomial curvature invariants are  constant across the manifold. In 
  \cite{CharBySpi} the authors proved that a four-dimensional Lorentzian manifold  with constant scalar curvature invariants,  is either locally homogeneous or a Kundt spacetime. In \cite{CSI}, \cite{CSI3} and \cite{CSI4} they found necessary and sufficient equations for Kundt metrics in dimension three and four to be $CSI.$ This classification shows examples of Lorentzian CSI metric which do not contain a single local Killing vector field, in sharp contrast to the Riemannian signature, where any CSI metric is automatically locally homogeneous. 

Locally homogeneous metrics can be characterized by having a locally transitive collection of Killing vector fields. In \cite{herviknew} the author, in an attempt to give a similar characterization of CSI metrics, gave a generalization of Killing vector fields, the nil-Killing vector fields, defined by requiring that the Lie derivative of the metric gives a nilpotent operator. It was noted that any Kundt vector field lies within this class. The nil-Killing vector fields were further studied in \cite{IDIFF} and \cite{NilKilling}, showing that under the assumption of algebraic preservation, they constitute a Lie algebra. It was seen that any CSI metric has a locally transitive collection of nil-Killing vector fields due to the local homogeneity of the transverse metric.

In this paper we use the flow properties of Kundt and nil-Killing vector fields in order to construct natural $G$-structures $P\rightarrow M$, for which they are realized as infinitesimal automorphisms of $P.$ These $G$-structure are shown to have a number of intrinsic properties. They give rise to an algebraic classification of tensors allowing for full contractions of even ranked tensors of type $II.$ In addition they have a natural class of metrics associated to them which constitute an affine space over symmetric rank two tensors of type $III.$ 

We go on to define \emph{Kundt structures} by an integrability criterion and requiring existence of certain local infinitesimal automorphisms. The collection of metrics belonging to such $G$-structures are automatically contained in the Kundt class.

Thus we obtain a machinery that allows for natural classes of metrics up to type $III$ tensors in such a way that even-ranked type $II$ tensors play a prominent role and infinitesimal automorphisms have analogous properties with the Kundt vector fields. Since all constructed deformations leaving the scalar polynomial invariants fixed have been achieved by deforming the metric in the direction of type $III,$ we hope that these $G$-structures can be employed to understand deformations of metrics, for which the curvature tensor and all its covariant derivatives are of type $II.$ 

Motivated by the characterization of CSI metrics through the application of nil-Killing vector fields, we characterize all $G$-invariant Kundt structures on homogeneous manifolds $G/H$. The idea is to present a CSI metric by a homogeneous space, in such a way that the left multiplication maps have the properties given by the flows of nil-Killing vector fields in analogy with representing a homogeneous Lorentzian manifold as the quotient of the isometry group by the isotropy group.
    
\section{G-structures}
Here we briefly present the notion of $G$-structures given in \cite{KN1} and \cite{KN2}. Suppose that $M$ is an $n$-dimensional manifold and $G\subset GL(n,\mathbb{R})$ is a Lie group. Recall that a $G$-structure on $M$ is a sub-principal bundle $P\overset{\pi}{\rightarrow}M$ with structure group $G$ of the principal frame bundle. Hence given $x\in M$ each fibre $\pi^{-1}(x)$ is a subset of $GL(\mathbb{R}^{n},T_{x}M)$ which is invariant under the right action of $G$ and on which $G$ acts freely.

If $f$ is a diffeomorphism on $M$ then $f$ induces an automorphism of the frame bundle by letting its derivative act on frames. $f$ is said to be an automorphism of the $G$-structure $P\overset{\pi}{\rightarrow} M$ if the induced automorphism of the frame bundle maps $P$ into itself. We let $Aut(P)$ denote the group of such diffeomorphisms.

A vector field $X\in \mathcal{X}(M)$ with flow $\phi_{t}$ is said to be an infinitesimal automorphism of $P\overset{\pi}{\rightarrow} M$ if $\phi_{t}\in Aut(P),$ for all $t.$

\section{The structure group}
In this section we shall discuss the structure group $GN$, which will be considered in the rest of the paper. The elements of the group $GN$ are linear transformations on Minkowski space which are characterized by two properties. They preserve the algebraic structure given by some fixed null-line $\lambda$, and arbitrary full contractions are preserved when acting on even-ranked tensors which are of type $II$ with respect to the $\lambda$. 
\newline

Consider $\mathbb{R}^{n}$ endowed with the Minkowski inner-product
\begin{equation}
    \eta= \omega^{1}\otimes\omega^{2} + \omega^{2}\otimes\omega^{1} +\omega^{3}\otimes\omega^{3}+\cdots + \omega^{n}\otimes\omega^{n},
\end{equation}
where $\{\omega^{1},\dots \omega^{n}\}$ is the dual to the standard basis $\{e_{1},\dots e_{n}\}.$ In order to follow convention we shall rename the basis elements by letting $\{k,l,m_{1}\dots m_{n-2}\}=\{e_{1},\dots e_{n}\}$.

We let $GN\subset GL(n,\mathbb{R})$ denote the Lie group of invertible linear transformations $f$ satisfying the following:
\begin{enumerate}[i)]
    \item $f(k)\in \mathbb{R}k,$
    \item $\eta(f(w),f(\tilde{w}))=\eta(w,\tilde{w}),\forall w,\tilde{w}\in \{k\}^{\perp},$
    \item $\eta(f(k),f(z))=\eta(k,z),\forall z\in \mathbb{R}^{n}.$
\end{enumerate}
Note in particular that $GN$ is not contained in $\mathcal{O}(1,n-1).$ 

The elements of $GN$ can be written as matrices of the form
\begin{equation}
    a=
    \begin{bmatrix}
    a & b_{1} & a_{1} &\dots & a_{n-2}\\
    0 & a^{-1} & 0 & \dots & 0 \\
    0 & b_{2} & c_{11} & \dots & c_{1(n-2)}\\
    \vdots & \vdots & \vdots & \ddots & \vdots\\
    0 & b_{n-1} & c_{(n-2)1} & \dots & c_{(n-2)(n-2)}
    \end{bmatrix},
\end{equation}
where $a\neq0,$ $a_{i},b_{i}$ are arbitrary numbers and $[c_{ij}]$ is an orthogonal matrix.

Letting $\lambda$ be the null-line spanned by $k$, we have the following characterizations which is a special case of a result from \cite{NilKilling} in section 2:
\begin{proposition}
Suppose that $f\in Gl(n,\mathbb{R})$, then the following are equivalent:
\begin{enumerate}[i)]
\item $f\in GN$.
\item $f(\lambda)\subset \lambda$ and $f^{*}\eta -\eta$ is of type $III$ w.r.t. $\lambda.$
\item The induced map $f^*$ on tensors preserves algebraic type and for any given even ranked tensor $T$ of type $II$ w.r.t. $\lambda,$ full contractions are preserved, i.e., \begin{equation}
    Tr(f^{*}T)=Tr(T).
\end{equation}
\end{enumerate}
\end{proposition}

In particular we see from \cite{NilKilling} that if $f\in GN$ and $T$ is any even-ranked type $II$ tensor, then the metrics $\eta$ and $f^{*}\eta$ induce the same full contractions of $T.$

\section{GN-structures}\label{KSGNStruct}

In this section we shall study $GN$-structures, and show that they induce an algebraic classification of tensors in analogy with the boost-order classification induced by a Lorentzian metric with a null-distribution \cite{AlgClass}. Moreover, we shall see that each $GN$-structure gives rise to a collection of metrics which forms an affine space over symmetric rank $2$ tensors of type $III$ in the algebraic classification.
\newline

We start by showing that to each $GN$-structure, $P\overset{\pi}{\rightarrow} M$, we can associate two distributions $\lambda,\Lambda$ on $M$ of dimension $1$ and codimension $1,$ respectively.

Given $x\in M$ and $u,\tilde{u}\in \pi^{-1}(x),$ there exists an element $a\in GN$ such that $\tilde{u}=ua.$ Since $a$ leaves the subspaces $\mathbb{R}k$ and $\{k\}^{\perp}$ invariant it follows that $$\tilde{u}(\mathbb{R}k)=ua(\mathbb{R}k)=u(\mathbb{R}k)$$ and $$\tilde{u}(\{k\}^{\perp})=ua(\{k\}^{\perp})=u(\{k\}^{\perp}).$$ Thus at each point of $x\in M$ we have a two well-defined subspaces 
\begin{equation}
    \lambda_{x} :=\{u(\mathbb{R}k):u\in \pi^{-1}(x)\}\subset T_{x}M,
\end{equation}
\begin{equation}
    \Lambda_{x} :=\{u(\{k\}^{\perp}):u\in \pi^{-1}(x)\}\subset T_{x}M,
\end{equation}
of dimension one and codimension one respectively. We let $\lambda$ and $\Lambda$ be the distributions defined by these subspaces at each point in $M,$ and refer to them as the distributions associated with the $GN$-structure $P\overset{\pi}{\rightarrow} M$. 

On the cotangent bundle we let $\lambda^{*}$ and $\Lambda^{*}$ be the annihilators of $\Lambda$ and $\lambda$ respectively, i.e.,
\begin{equation}
    \lambda^{*}_{x}=\{\omega\in T^{*}_{x}M:\omega(X)=0, \, \forall \, X\in \Lambda_{x}\}
\end{equation}
and
\begin{equation}
    \Lambda^{*}_{x}=\{\omega\in T^{*}_{x}M:\omega(X)=0, \, \forall \, X\in \lambda_{x}\}.
\end{equation}
Using this we give an algebraic classification of tensors as follows: Let 
\begin{equation}
TM^1=TM,\quad TM^0=\Lambda,\quad TM^{-1}=\lambda
\end{equation}
and
\begin{equation}
    T^{*}M^{1}=TM^*,\quad T^{*}M^{0}=\Lambda^{*},\quad T^{*}M^{-1}=\lambda^{*}.
\end{equation}
Let $\mathcal{D}^{r}_{t}(M)$ denote the $C^{\infty}(M)$-module of tensors with $r$ covariant and $t$ contravariant factors. Following standard terminology from \cite{AlgClass}  we define the sub-module of tensors with boost-order $s$ by
\begin{equation}\label{KSBoostorder}
    \bigoplus_{\substack{s_{1}+\cdots +s_{r+t}=s\\s_{i}\in \{-1,0,1\}}}T^{*}M^{s_{1}}\otimes \cdots\otimes T^{*}M^{s_{r}}\otimes TM^{s_{r+1}}\otimes TM^{s_{r+k}}.
\end{equation}
Tensors of boost-order $0$ and $-1$ are said to be of type $II$ and $III$ respectively. Thus $GN$-structures induce an algebraic boost-order classification of tensors similar to that of a Lorentzian manifold with a given null-distribution.

Next we proceed to show that each $GN$-structure induces a natural collection of Lorentzian metrics. Recall that the Lie subgroup $Sim(n-2)$ of $\mathcal{O}(1,n-1)$ is given by
\begin{equation}
    Sim(n-2):=\{f\in \mathcal{O}(1,n-1):f(k)\in \mathbb{R}k\}.
\end{equation}
We see that $Sim(n-2)$ is a Lie subgroup of $GN,$ satisfying the relation
\begin{equation}
    Sim(n-2)=GN\cap\mathcal{O}(1,n-1).
\end{equation}
Each $Sim(n-2)$-structure $Q\rightarrow M$ induces a metric on $M$ by letting the elements of $Q$ be null-bases for $g$. Using this we have the following definition.

\begin{definition}
Let $P\rightarrow M$ be a $GN$-structure on $M.$ A metric $g$ on $M$ is said to belong to $P$ if it is determined by a $Sim(n)$-subprincipal bundle of $P.$ The collection of such metrics is denoted by $\mathcal{M}.$
\end{definition}

\begin{characterisation}\label{metricsbelonging}
Let $P\rightarrow M$ be a $GN$-structure with associated distributions $\lambda$ and $\Lambda$.  A metric $g$ belongs to $P$ if and only if with respect to $g$ the following are satisfied:
\begin{enumerate}[i)]
    \item $\lambda$ is a null distribution.
    \item $\lambda^{\perp}=\Lambda.$
    \item If $(X,Y,Z_{i})\in P$, then $g(X,Y)=1$ and $g(Z_{i},Z_{j})=\delta_{ij},$ for all $i,j.$
\end{enumerate}
\end{characterisation}
\begin{proof}
Suppose first that $g$ satisfies the three conditions $i)-iii)$. By $i)$, $\lambda$ is a null-distribution with respect to this metric. Let $Q\rightarrow M$ be the $Sim(n)$-structure defined as the collection of null-frames with respect to $g$ whose first frame element points along $\lambda.$ The proof of the implication will be completed if we show that $Q\subset P.$ Now given a point $p\in M$ and an element $(X,Y,Z_{i})$ in $P$ in the fibre above $p$, then $ii)$ and $iii)$ imply that we can find a vector $\tilde{Y}$ such that $(X,\tilde{Y},Z_{i})$ is a null-frame for $g$ and $Y=cX+\tilde{Y}+d^{k}m_{k},$ for some coeffisients $d^{k}$, for $k=1\dots n-2$. Now define $a\in GL(\mathbb{R}^{n})$ by its action on the standard null-frame $\{k,l,m_{i}\}$ on $\mathbb{R}^{1,n-1}$ by $$a(l)=ck+l +d^{i}m_{i},\quad a(k)=k, \quad a(m_{i})=m_{i},$$ for all $i.$ Then clearly $a\in GN$ and $(X,\tilde{Y},Z_{i})a=(X,Y,Z_{i})$. It follows that $(X,\tilde{Y},Z_{i})\in P$ and therefore the fibre of $Q$ above $p$ is contained in $P$, which shows that $Q\subset P,$ thus the metric belongs to $P.$
    
\end{proof}

\begin{proposition}\label{KSaffine}
Given a $GN$-structure $P\rightarrow M$, the collection $\mathcal{M}$ of metrics belonging to $P$ is a non-empty affine space with respect to covariant symmetric rank $2$ tensors of type $III$.
\end{proposition}
\begin{proof}
Suppose that $g\in \mathcal{M}$ and $T$ is a symmetric two-tensor of type $III.$ Let $\lambda$ and $\Lambda$ be the distributions induced by $P$. By definition we see that $T\in \lambda^{*}\otimes_{s}\Lambda^{*}.$ It follows easily that $\lambda$ is a null-distribution for the metric $g+T$ such that $\lambda^{\perp}=\Lambda.$ Furthermore, if $(X,Y,Z_{i})\in P$, then $T(X,Y)=0$ since $X\in\lambda$. Moreover ${Z_{i}}\subset \Lambda$ and hence $T(Z_{i},Z_{j})=0,$ for all $i,j$. Thus by the above characterisation $g+T$ belongs to $P.$

Now suppose that $\tilde{g}\in \mathcal{M}$. If $(X,Y,Z_{i})\in P,$ then
\begin{equation}
    (\tilde{g}-g)(X,Y)=\tilde{g}(X,Y)-g(X,Y)=0,
\end{equation}
and 
\begin{equation}
     (\tilde{g}-g)(Z_{i},Z_{j})=\tilde{g}(Z_{i},Z_{j})-g(Z_{i},Z_{j})=\delta_{ij} -\delta_{ij}.
\end{equation}
It follows readily that $\tilde{g}-g\in \lambda^{*}\otimes_{s}\Lambda^{*}$ and hence is a tensor of type $III.$

Now let us show that the collection $\mathcal{M}$ of metrics belonging to a $GN$-structure $P\rightarrow M$ is always non-empty. Given a point $p\in M$ we can find a neighborhood $U$ and a local section $s:U\rightarrow P.$ Let $g$ be the unique Lorentzian metric on $U$ such that $s(q)$ is a null-frame for $g,$ for all $q\in U.$ Then $g$ belongs to the $GN$-structure given by the restriction of $P$ to $U.$

Proceeding in this way we can find a locally finite covering $\{U_{\alpha}\}_{\alpha \in I}$ of $M$ and metrics $g_{\alpha}$ on $U_{\alpha}$ belonging to the restriction of $P$ to $U_{\alpha},$ for all $\alpha\in I.$ Now let ${\psi_{\alpha}}_{\alpha\in I}$ be a partition of unity subordinate to $\{U_{\alpha}\}_{\alpha \in I}$. Since any two metrics $g_{\alpha},g_{\beta}$ differ by a tensor of type $III$ on a non-empty intersection $U_{\alpha}\cap U_{\beta},$ it follows that 
\begin{equation}
    \sum_{\alpha\in I}\psi_{\alpha}g_{\alpha}
\end{equation}
gives a well-defined Lorentzian metric which belongs to $P.$ This finishes the proof.
\end{proof}

We have another characterisation of $GN$-structures in terms of certain quadruples which will be needed later. On a manifold $M$ of dimension $n$ a \emph{null-quadruple} is defined to be a quadruple $(\lambda,\Lambda,g^{\perp},f)$ consisting of the following datum: 
\begin{enumerate}[i)]
    \item Two distributions $\lambda$ and $\Lambda$ of dimension $1$ and $n-1$ respectively such that $\lambda\subset \Lambda.$
    \item A symmetric positive-definite bilinear form $g^{\perp}$ on the quotient bundle $\Lambda/\lambda.$
    \item A non-vanishing one-form $f$ on the bundle $\lambda\otimes TM/\Lambda.$
\end{enumerate}

\begin{proposition}
On a manifold $M$ there is a $1:1$ correspondence  between the collection of $GN$-structure $P\rightarrow M$ and null-quadruples $(\lambda,\Lambda,g^{\perp},f).$ 
\end{proposition}
\begin{proof}
If $P\rightarrow M$ is a $GN$-structure on $M$, let $\lambda$ and $\Lambda$ be the associated distributions. If $p\in M$ then taking any element $(X,Y,Z_{i})\in P$ above $p$ we see that $(X,Z_{i})$ determines a degenerate symmetric bilinear form $\beta$ on $\Lambda_{p}$ by setting $\beta(X,X)=0$ and $\beta(Z_{i},Z_{i})=1.$ Due to the form of $GN$ this gives a well-defined degenerate symmetric two-form on $\Lambda$ which is independent of choice of elements in $P.$ Moreover $\beta$ induces a positive definite symmetric bilinear form $g^{\perp}$ on the quotient bundle $\Lambda/\lambda$. 

Given $p\in M,$ we define the map $f_{p}:\lambda_{p}\otimes T_{p}M/\Lambda_{p}\rightarrow\mathbb{R}$ as follows: Let $(X,Y,Z_{i})$ be an element of $P$ in the fibre above $p$ whose first entry is $X.$ Then we set \begin{equation}f_{p}(X\otimes[Y])=1,\end{equation} where $[Y]$ is the coset of $Y$ in $T_{p}M/\Lambda_{p}$, and extent $f_p$ to $\lambda_{p}\otimes T_{p}M/\Lambda_{p}$ by linearity. This map is well-defined since if $(\tilde{X},\tilde{Y},\tilde{Z}_{i})$ is another element of $P$, then due to the form of $GN$ we see that $\tilde{X}=bX$ $\tilde{Y}=cX+b^{-1}Y+d^{i}Z_{i}$, for some $b,c,d^{i}\in \mathbb{R}$ which implies that $\tilde{X}\otimes[\tilde{Y}]=X\otimes[Y].$ Hence $(\lambda,\Lambda,g^{\perp},f)$ defines a null-quadruple. 

Conversely, suppose we have a null-quadruple $(\lambda,\Lambda,g^{\perp},f)$. At each point $p\in M$ define the $P_{p}$ to be the collection of bases $(X,Y,Z_{i})$ for $T_{p}M$ such that $X\in \lambda_{p}$, $f(X\otimes[Y])=1$ and the classes of $Z_{i}$ in $\Lambda_{p}/\lambda_{p}$ form an orthonormal basis with respect to $g^{\perp}$. Now if $a\in GN$ and $(X,Y,Z_{i})\in P_{p}$ , the letting $(\tilde{X},\tilde{Y},\tilde{Z}_{i}):=(X,Y,Z_{i})a,$ we know that $$(\tilde{X},\tilde{Y},\tilde{Z}_{i})=(bX,cX+b^{-1}Y+d^{j}Z_{j},f^{i}k + \tensor{S}{_i^m}Z_{m}),$$
where $b,c,d^{k},f^{k}$ are constants and $\tensor{S}{_i^m}$ is an orthogonal $(n-2)\times(n-2)$ matrix. Clearly \begin{equation}
    f(\tilde{X}\otimes[\tilde{Y}])=f((bX)\otimes[cX+b^{-1}Y+d^{j}Z_{j}])=f(X\otimes[Y])=1,
\end{equation} and since $\tensor{S}{_{i}^m}$ is orthogonal, the classes of $\tilde{Z}_{i}$ are again an orthonormal basis for $g^{\perp}$. Thus $(\tilde{X},\tilde{Y},\tilde{Z}_{i})\in P_{p},$ from which it follows that $P_{p}$ is invariant under the right action of $GN.$ 

Let us show that $GN$ acts transitively on $P_{p}$. Given any other element $(X^{\prime},Y^{\prime},Z_{i}^{\prime})$, then by injectivity of $f$ it follows that $X^{\prime}\otimes[Y^{\prime}]=X\otimes[Y]$ which implies that $X^{\prime}=bX$ and $Y^{\prime}=cX+b^{-1}Y+d^{j}Z_{j}$ for some constants $b,c,d^{i}$. Using this and the fact that $Z^{\prime}_{i}$ also gives an orthonormal basis for $g^{\perp}$ it is clear that we can construct $a\in GN$ such that $(X^{\prime},Y^{\prime},Z^{\prime}_{i})=(X,Y,Z_{i})a.$ Thus it follows that $\cup_{p\in M}P_{p}$ is a GN-structure. 

Clearly the correspondence between $GN$-structures and null-quadruples is a bijection.
\end{proof}

 It is useful to characterize automorphisms and metrics belonging to $GN$-structures through this correspondence. Suppose therefore that $P\rightarrow M$ is a $GN$-structure corresponding a null-quadrupe $(\lambda,\Lambda,g^{\perp},f)$. 
 
 A diffeomorphism $\phi:M\rightarrow M$ is an automorphism of $P$ if and only if $\phi_{*}$ satisfies the following:
 \begin{enumerate}[i)]
     \item $\phi_{*}(\lambda)=\lambda$ and $\phi_{*}(\Lambda)=\Lambda,$
     \item $\phi^{*}g^{\perp}=g^{\perp},$
     \item $\phi^{*}f=f.$
 \end{enumerate}
 where the pull-backs in $ii)$ and $iii)$ are well-defined due to $i)$.
 
Furthermore it is clear from characterization \ref{metricsbelonging} that a metric $g$ belongs to $P$ if and only if 
\begin{enumerate}[i)]
    \item $g(W,W^{\prime})=g^{\perp}([W],[W^{\prime}])$, for all $W,W^{\prime}\in \Lambda$,
    \item $g(X,Y)=f(X\otimes[Y]),$ for all $X\in \lambda$ and $Y\in TM$.
\end{enumerate}
 
 Let us now discuss some general features of this framework. If $P\rightarrow M$ is a $GN$-structure corresponding to a null-quadruple $(\lambda,\Lambda,g^{\perp},f).$ If a metric $g$ belongs to $P,$ then since $\lambda$ is a null-distribution w.r.t. $g$ and $\lambda^{\perp}=\Lambda,$ it follows that the algebraic boost-order classification of tensors induced by $(M,g,\lambda)$ is the same as the algebraic classification derived from $P\rightarrow M$ given above in \eqref{KSBoostorder}.
 
 Furthermore, if $\tilde{g}$ is another metric belonging to $P,$ then by proposition \ref{KSaffine}, $\tilde{g}-g$ is a tensor of type $III$ w.r.t. $\lambda.$ Now suppose that $T$ is an even ranked type $II$ tensor. Then any full contraction of $T$ perfomed with respect to either $g$ or $\tilde{g}$ must be the same, since any contribution of a tensor of type $III$ must vanish.
 
 Hence any $GN$-structure $P\rightarrow M$  gives well-defined full contractions
 \begin{equation}
     Tr(T),
 \end{equation}
 of any even-ranked tensor $T$ of type $II.$
 
 Lastly, if $X$ is a vector field belonging to the null-distribution $\lambda,$ then using the functional $f\in(\lambda\otimes TM/\Lambda)^{*},$ the $GN$-structure allows us to define a dual form $X^{\natural}$ by 
 \begin{equation}
     X^{\natural}(Z)=f(X\otimes [Z]),\quad \forall\, Z\in \mathcal{X}(M).
 \end{equation}
 We see that $X^{\natural}$ vanishes on $\Lambda,$ showing that $X^\natural\in \lambda^{*}.$ Clearly if $g\in \mathcal{M},$ then this dual coincides with the metric dual for vector fields belonging to $\lambda.$
 \newline
 
 Next we consider connections on $GN$-structures:
 \begin{proposition}\label{KSConnections}
     Let $P\rightarrow M$ be a $GN$-structure corresponding to a null-quadruple $(\lambda,\Lambda,g^{\perp},f)$ and suppose that $g$ is a metric belonging to $P.$ If $\Gamma$ is a linear connection on $M$ with affine connection $\nabla,$ then 
     the following are equivalent:
     \begin{enumerate}[i)]
     \item $\Gamma$ is a connection in $P$.
     \item $\nabla_{Z} \lambda\subset \lambda$, $\nabla_{Z}\Lambda\subset \Lambda$ and $\nabla_{Z}g$ is of type $III,$ for all vector fields $Z.$
     \item $\nabla$ restricts to an affine connection on $\lambda$ and $\Lambda$ such that $g^\perp$ and $f$ are covariantly constant w.r.t. the connections induced on $\Lambda/\lambda$ and $\lambda \otimes TM/\Lambda$ respectively.
     \end{enumerate}
 \end{proposition}
\begin{proof}
$"i)\Leftrightarrow ii)"$
Let $\{U_{\alpha}\}_{\alpha\in I}$ be a covering with local frames $\{k^{\alpha},l^{\alpha},m_{i}^{\alpha}\}_{\alpha\in I}$ of $P,$ such that $\{k^{\alpha},l^{\alpha},m_{i}^{\alpha}\}_{\alpha\in I}$ are null-frames for $g.$ Then $\Gamma$ is a connection on $P$ if and only if each local connection one-form
$\{A_{\alpha}\}_{\alpha\in I}$ has values in the Lie algebra $\mathfrak{gn}$ of $GN.$ 

Suppose that $\Gamma$ is a connection in $P$. If $Z\in \mathcal{X}(M),$ then
\begin{equation}
\begin{gathered}
    (\nabla_{Z}g)(k^{\alpha},l^{\alpha})=-g(\nabla_{Z}k^{\alpha},l^{\alpha})-g(k^{\alpha},\nabla_{Z}l^{\alpha})\\=-g(\tensor{(A_{\alpha}(Z))}{^1_1}k^{\alpha},l^{\alpha})-g(k^{\alpha},\tensor{(A_{\alpha}(Z))}{^2_2}l^{\alpha})=0,
\end{gathered}
\end{equation}
where the last equality holds since $\tensor{(A_{\alpha}(Z))}{^1_1}=-\tensor{(A_{\alpha}(Z))}{^2_2}.$ By the same reasoning $\nabla_{Z}g(m^{\alpha}_{i},m^{\alpha}_{j})$ =0, since $\tensor{A(Z)}{^{i}_{j}}$ is skew-symmetric for $i,j\geq 3.$ It therefore follows that $\nabla_{Z}g$ is of type $III$ and $\nabla_{Z}\lambda\subset \lambda$ and $\nabla_{Z}\Lambda \subset \Lambda.$ In order to prove the converse one follows the same reasoning to show that $ii)$ implies that $\tensor{A(Z)}{^{i}_{j}}\in \mathfrak{gn}.$

$ii)$ and $iii)$ are clearly equivalent since $g$ induces $g^{\perp}$ and $f$ on $\Lambda/\lambda\otimes \Lambda/\lambda$ and $\lambda\otimes TM/\Lambda$ respectively.
\end{proof}

\section{GN-structures and Nil-Killing vector fields}

Recall from \cite{herviknew,IDIFF} that if $(M,g)$ is a Lorentzian manifold  then a vector field $X$ is said to be nil-Killing if $\tensor{(\mathcal{L}_{X}g)}{^{a}_{b}}$ is nilpotent. Given a point $p\in M,$ one can shows that the operator corresponding to $\mathcal{L}_Xg$ is nilpotent if and only if there is some null-line $\lambda_{p}\subset T_{p}M$ such that $(\mathcal{L}Xg)_{p}$ is of type $III$ w.r.t. $\lambda_{p}.$ Moreover, this null-line is unique provided that $X$ is not Killing.

Therefore if $\lambda$ is a null-distribution on $M$ we say that a vector field $X$ is nil-Killing with respect to $\lambda$ if $\mathcal{L}_{X}g$ is of type $III$ w.r.t. $\lambda.$ One can see that this is equivalent to 
\begin{equation}
    \mathcal{L}_{X}g(Y,Z)=0, \, \mathcal{L}_{X}g(W,\tilde{W})=0,
\end{equation}
for all $Y\in \lambda,$ $Z\in TM$ and $W,\tilde{W}\in \{\lambda\}^{\perp}.$

If in addition $[X,\lambda]\subset\lambda$, then $X$ preserves algebraic structure in the sense that if $\phi_{t}$ is the flow of $X,$ then $\phi_{t}$ preserves the boost-order of tensors, given by $\lambda.$ We shall refer to such vector fields as algebra preserving nil-Killing vector fields w.r.t. $\lambda$. It was seen in \cite{IDIFF} that the collection of such vector fields
\begin{equation}
    \mathfrak{g}_{(g,\lambda)} :=\{\text{Algebra perserving nil-Killing vector fields w.r.t. } \lambda\}
\end{equation}
constitute a Lie algebra. From \cite{NilKilling} we have the following characterization:
\begin{proposition}\label{flowchar}
Let $(M,g,\lambda)$ be a Lorentzian manifold with a null vector distribution $\lambda.$ A vector field $X$ is an algebra preserving nil-Killing w.r.t. $\lambda$ iff. for each $t$ the flow $\phi_{t}$ of $X$ satisfies the following:
\begin{enumerate}[i)]
    \item $(\phi_{t})_{*}(\lambda)\subset\lambda.$
    \item $\phi_{t}^{*}g-g$ is of type $III$ w.r.t. $\lambda.$
\end{enumerate}
\end{proposition}
Hence we see that such vector fields have the same characteristics as the elements in the structure group $GN$. Indeed we have the following result:
\begin{proposition}\label{nil}
Suppose that $P\rightarrow M$ is a $GN$-structure, and let $\lambda$ be its null-distribution. Given a metric $g$ belonging to $M$, then a vector field $X$ is an infinitesimal automorphism of $P$ iff. $X$ is an algebra preserving nil-Killing vector field w.r.t $(M,g,\lambda).$
\end{proposition}
\begin{proof}
Let $(\lambda,\Lambda,g^{\perp},f)$ be the null-quadruple corresponding to $P\rightarrow M$ and $\phi_{t}$ the flow of $X$.

By proposition \ref{flowchar}, $X$ is an algebra preserving nil-Killing vector field w.r.t. $(M,g,\lambda)$ iff. \begin{equation}(\phi_{t})_{*}(\lambda)\subset \lambda,  \end{equation}\begin{equation}g((\phi_{t})_{*}(\tilde{Z}),(\phi_{t})_{*}(Y))=g(\tilde{X},Y),\quad g((\phi_{t})_{*}(W_1),(\phi_{t})_{*}(W_2))=g(W_{1},W_{2}),\end{equation} for all $t,$ $Z\in \lambda$, $Y\in TM$  and $W_1,W_2\in \lambda^{\perp}.$ 

Since $g$ belongs to $P$ this holds iff. \begin{equation}(\phi_{t})_{*}(\lambda)\subset \lambda,\quad (\phi_{t})_{*}(\Lambda)\subset \Lambda,\end{equation} \begin{equation} f((\phi_{t})_{*}Z\otimes[(\phi_{t})_{*}Y])=f(Z,Y),\quad g^\perp((\phi_{t})_{*}(W_1),(\phi_{t})_{*}(W_2))=g^\perp(W_{1},W_{2}),\end{equation} for all $t,$ $Z\in \lambda$, $Y\in TM$  and $W_1,W_2\in \lambda^{\perp}.$

By the characterisation of automorphisms and  metrics belonging to $P$ in terms of null-quadruples given in section \ref{KSGNStruct}, this is true iff. $X$ is an infinitesimal automorphism of $P.$
\end{proof}

If $(M,g,\lambda)$ is a Lorentzian manifold with a null-distribution, then we obtain an associated null-quadruple given by $(\lambda,\lambda^{\perp},g^{\perp},f),$ where $g^{\perp}$ and $f$ are induced by $g$ by taking in a natural manner. We denote the corresponding $GN$-structure by \begin{equation}P(M,g,\lambda).\end{equation} By construction it is clear that the metric $g$ belongs to $P$.

\begin{corollary}\label{autalg}
Let $(M,g,\lambda)$ be a Lorentzian manifold with a null-distribution, and let $P(M,g,\lambda)$ be the associated $GN$-structure. Then
\begin{equation}
\mathfrak{g}_{(g,\lambda)}=aut[P(M,g,\lambda)],
\end{equation}
where $aut[P(M,g,\lambda)]$ is the Lie-algebra of infinitesimal the $GN$-structure.
\end{corollary}

\medskip

\section{Kundt structures}\label{KSKundtStructuresSection}

In this section we shall see use proposition \ref{nil} to give conditions that ensure that the metrics belonging to a $GN$-structure are Kundt spacetimes with respect to the null-distribution associated to the $GN$-structure.

First let us recall the definition of a Kundt spacetime \cite{KundtSpacetimes,GenKundt,ExactSolutions}. Let $(M,g,\lambda)$ be a Lorentzian manifold with a null-distribution $\lambda.$ Such a triple is said to be Kundt if $\lambda^{\perp}$ is integrable and for each point $p\in M$ there exists a vector field $X$ belonging to $\lambda$ such that $X$ is affinely geodesic, shear-free and divergence-free, i.e.,
\begin{equation}\label{KSGeoShearDiv}
\nabla_{X}X=0,\quad \nabla_{(a}X_{b)}\nabla^{a}X^{b}\quad \text{ and } \quad \nabla_{a}X^{a}=0,   \end{equation}
respectively. In this case we say that $X$ is a Kundt vector field of $(M,g,\lambda)$.

Without assuming integrability, it was seen in \cite{NilKilling} that a vector field $X$ belonging to a null-distribution $\lambda$ satisfies \eqref{KSGeoShearDiv} if and only if
$X$ is nil-Killing w.r.t $\lambda.$ Since $X$ belongs to $\lambda$ such a vector field automatically preserves the algebraic structure given by $\lambda.$

The following definition give an analogue of this conditions for $GN$-structures.

\begin{definition}
A $GN$-structure $P\rightarrow M$ with associated null-distributions $\lambda$ and $\Lambda$ is said to be a Kundt structure if $\Lambda$ is integrable and about each point $p\in M$ there exists a local vector field $X$ belonging to the distribution $\lambda$ which is an infinitesimal automorphism of $P.$ 
\end{definition}
Now the following results on metrics belonging to Kundt structures are corollaries of proposition \ref{nil}.

\begin{corollary}\label{KSKundtCorollary}
If $P\rightarrow M$ is a Kundt structure with associated null-distribution $\lambda,$ then for each metric $g$ belonging to $P,$ the triple $(M,g,\lambda)$ is a Kundt spacetime.
\end{corollary}
\begin{proof}
If $p\in M$, then since $P\rightarrow M$ is a Kundt structure we can find a vector field $X$ in a neighborhood about $p$ belonging to $\lambda$ such that $X$ is an infinitesimal automorphism of $P.$ It follows from proposition \ref{nil} that $X$ is nil-Killing w.r.t. $(M,g,\lambda).$ By assumption $\lambda^{\perp}=\Lambda$ is integrable, and therefore $(M,g,\lambda)$ is a Kundt spacetime.
\end{proof}

\begin{corollary}
If $(M,g,\lambda)$ is a Kundt spacetime, then the induced $GN$-structure $P(M,g,\lambda)$ is a Kundt structure.
\end{corollary}
\begin{proof}
This follows from the fact the the collection of local nil-Killing vector fields of $(M,g,\lambda)$ is equal to the collection of local infinitesimal automorphisms of $P.$
\end{proof}

The following proposition gives a characterization of $Kundt$ structures in terms of connections. \begin{proposition}
If $P\rightarrow M$ is a $GN$-structure, then the following are equivalent:
\begin{enumerate}[i)]
    \item $P$ is a Kundt structure.
    \item For each point $p\in M$ there is a neighborhood $U$ such that the restriction of $P$ to $U$ has a torsion-free connection.
    \end{enumerate}
\end{proposition}
\begin{proof}
$"ii)\Rightarrow i)"$ 
First we show that $\Lambda$ is integrable. Let $U$ be an open set with a torsion-free connection $\nabla$ belonging to the restriction of  $P$ to $U.$ If $W_{1},W_{2}\in \mathcal{X}(U)$ belong to $\Lambda,$ then  $\nabla_{W_i}W_{j}\in \Lambda,$ for $i,j=1,2,$ by proposition \ref{KSConnections}. Therefore since $\nabla$ is torsion-free this implies that
\begin{equation}[W_{1},W_{2}]=\nabla_{W_{1}}W_{2}-\nabla_{W_{2}}W_{1}\in \Lambda.\end{equation}
Since we can cover $M$ by open sets having torsion-free connections belonging to $P,$ this shows that $\Lambda$ is an integrable distribution.

Now let us show that $P$ has the desired local infinitesimal automorphisms belonging to $\lambda.$ Suppose $p\in M$ and choose a neighborhood $U$ with a torsion-free connection $\nabla$ belonging to $P$ on $U.$ Since $\Lambda$ is integrable we can, by shrinking $U$ if necessary,  choose $X\in \lambda$ on $U$ such that $X^{\natural}$ is a closed one-form. By proposition \ref{KSConnections}, we see that $\nabla X^{\natural}=\omega\otimes X^{\natural}$, for some one-form $\omega$.

Since $\nabla$ is Torsion-free it follows that \begin{equation}Alt(\nabla X^{\natural})=dX^{\natural}=0,\end{equation} 
and therefore $\omega$ is proportional to $X^{\natural}$, from which we see that $\nabla X^{\natural}$ is of type $III.$

Now let $g$ be any metric belonging to $P.$ It follows again by proposition \ref{KSConnections} that $\nabla_{Z} g$ is a tensor of type $III,$ for all $Z\in \mathcal{X}(M)$. Let us show that if $Z\in \mathcal{X}(M),$ then $\nabla_{Z}(X^\natural)=(\nabla_{Z}X)^{\natural}$. Suppose that $Z^\prime\in \mathcal{X}(M),$ then
\begin{equation}
    \begin{gathered}
    (\nabla_{Z}X^{\natural})(Z^{\prime})=Z(X^{\natural}(Z^{\prime}))-X^{\natural}(\nabla_{Z}Z^{\prime})\\=Z(g(X,Z^{\prime}))-g(X,\nabla_{Z}Z^{\prime}) =(\nabla_{Z}g)(X,Z^{\prime}) + g(\nabla_{Z}X,Z^{\prime})\\
    (\nabla_{Z}X)^\natural(Z^{\prime}),
    \end{gathered}
\end{equation}
where the last equality holds since $\nabla_{Z}g$ is of type III.

If $Z,Z^{\prime}\in \mathcal{X}(M),$ then by the torsion-free property
\begin{equation}
\begin{gathered}
    \mathcal{L}_{X}g(Z,Z^{\prime})\\=Xg(Z,Z^{\prime})-g([X,Z],Z^{\prime})-g(Z,[X,Z^{\prime}])\\
    =Xg(Z,Z^{\prime}) - g(\nabla_{X}Z-\nabla_{Z}X,Z^{\prime})-g(Z,\nabla_{X}Z^{\prime}-\nabla_{Z^\prime}X)\\
    =(\nabla_Xg)(Z,Z^{\prime}) +(\nabla_{Z}X)^{\natural}(Z^{\prime})+ (\nabla_{Z^{\prime}}X)^{\natural}(Z)\\
    =(\nabla_Xg)(Z,Z^{\prime}) +(\nabla_{Z}X^{\natural})(Z^{\prime})+ (\nabla_{Z^{\prime}}X^{\natural})(Z),
\end{gathered}
\end{equation}
and therefore $\mathcal{L}_Xg=\nabla_{X}g+S(\nabla X^{\natural}),$ where $S$ denotes symmetrization. It follows that $\mathcal{L}_Xg$ is of type $III$ w.r.t the null-distribution $\lambda$ associated to $P,$ and therefore $X$ is nil-Killing w.r.t. $(M,g,\lambda)$. Thus by proposition \ref{nil}, $X$ is a local infinitesimal automorphism of $P,$ showing that $P$ is Kundt.

$"i)\Rightarrow ii)"$
Let $(\lambda,\Lambda,g^{\perp},f)$ be the null-quadruple associated to $P.$ By integrability we can find a neighborhood $U$ about $p$ with coordinates $(u,v,x^{k})$ such that $\frac{\partial}{\partial v}$ is an infinitesimal automorphism of $P$, $\frac{\partial}{\partial v}\in \lambda$, $\frac{\partial}{\partial x^{i}}\in \Lambda$ and $du=(\frac{\partial}{\partial v})^{\natural}$.  

Define a metric $g$ on $U$ by
\begin{equation}
    g=2dudv + \tilde{g}_{ij}(u,v,x^{k})dx^{i}dx^{j},
\end{equation}
in such a way that $\tilde{g}_{ij}$ is compatible with $g^{\perp}$. Then the metric $g$ belongs to $P$ on $U$ and since $\frac{\partial}{\partial v}$ is an infinitesimal automorphism of $P$ it follows that $\tilde{g}_{ij}$ is independent of $v.$ Therefore its Levi-civita connection is a Torsion-free connection belonging to $P$ over $U.$
\end{proof}
Regarding the curvature of a torsion-free connection belonging to a $GN$-structure, we have the following result:
\begin{proposition}
Let $P\rightarrow M$ be a $GN$-structure. If $\nabla$ is a torsion-free connection which belongs to $P,$ then its curvature tensor $R^{\nabla}$ is of type $II.$
\end{proposition}
\begin{proof}
Fix a metric $g$ belonging to $P$. Let $R$ be the covariant four-tensor defined by
\begin{equation}
    R(X,Y,Z,W)=g((R^{\nabla})(X,Y)Z,W).
\end{equation}
The connection $\nabla$ belongs to $P$ and therefore  its associated curvature two-form on $P$ is $\mathfrak{gn}$-valued, where $\mathfrak{gn}$ is the Lie algebra of $GN.$ Since in addition $g$ belongs to $P$ this implies that 
\begin{equation}\label{Rmskew}
    R(X,Y,Z,U) = -R(X,Y,U,Z)\quad \text{ and } \quad R(X,Y,W,W^{\prime}) = -R(X,Y,W^{\prime},W),
\end{equation}
for all $X,Y,U\in \mathcal{X}(M)$, $Z\in \lambda$ and $W,W^{\prime}\in \Lambda.$
By abuse of notation it is therefore clear that $Rm^{\nabla}$ is of type $II$ if and only if the terms
\begin{equation}\label{positivecomp}
  R(\lambda,\Lambda,\lambda,TM), \quad R(\Lambda,\Lambda,\lambda,\Lambda),\quad R(TM,\lambda,\lambda,\Lambda),
  \quad R(\Lambda,\lambda,\Lambda,\Lambda),
\end{equation}
vanish.

If $p\in M,$ choose a vector field $X\in \lambda$  in a neighborhood of $p$  such that $X^{\natural}$ is closed. Then there exists some smooth function $f$ such that $\nabla X=f(X^{\natural}\otimes X).$ Hence we see that if $W,W^{\prime}\in \Lambda$, then
\begin{equation}
    R^{\nabla}(W,W^{\prime})X=[\nabla_{W},\nabla_{W^{\prime}}]X-\nabla_{[W,W^{\prime}]}X=0.
\end{equation}
Therefore the two first expressions in \eqref{positivecomp} vanish. We can see that $R(TM,\lambda,\lambda,\Lambda)$ vanishes, from the fact that $\nabla \lambda \subset \lambda$ by proposition \ref{KSConnections}.

Lastly, suppose that $W,W^{\prime},U,U^{\prime}\in \Lambda$.
Since $\nabla$ is Torsion-free, we know that $Rm^{\nabla}$ satisfies the first Bianchi identity. Using \eqref{Rmskew} together with the first Bianchi identity one can show that
\begin{equation}
    R(W,W^{\prime},U,U^{\prime})=R(U,U^{\prime},W,W^{\prime}),
\end{equation}
by copying the proof of the similar statement in the Riemannian setting. In particular this shows that the last expression in \eqref{positivecomp} vanishes, finishing the proof.
\end{proof}
The above proposition shows that if $\nabla$ is a torsion-free connection of a $GN$-structure $P\rightarrow M,$ then its curvature tensor $R^\nabla$ is of type $II.$ By the results of section \ref{KSGNStruct} we can use the $GN$-structure to perform full contractions of tensors in the algebra generated by $R^{\nabla},$ giving smooth functions associated to $\nabla.$
As an example, this gives us a way to defined the scalar curvature of torsion-free connections belonging to $P.$
\newline

Next we show that if $P\rightarrow M$ is a Kundt-structure, then we have an induced map
\begin{equation}
    \Phi:\mathcal{M}\rightarrow \Omega^{0}(\Lambda/\lambda).
\end{equation}
from the space of metrics belonging to $P$ into the space of sections of $\Lambda/\lambda.$

If $X$ is a local infinitesimal automorphism of $P$ belonging to $\lambda$ such that $X^{\natural}$ is closed and $g\in \mathcal{M},$ then $\mathcal{L}_{X}g$ is of type $III$ and can therefore be written uniquely as
\begin{equation}
    \mathcal{L}_{X}g=X^{\natural}\otimes_{S}\omega,
\end{equation}
where $\omega\in \Lambda^{*}$. Now supposing that $\tilde{X}$ is another local infinitesimal automorphism such that $\tilde{X}^{\natural}$ is closed, then there exists a smooth function $f$ such that $\tilde{X}=fX$ and $W(f)=0,$ for all $W\in \Lambda, $ implying that $df\in \lambda^{*}.$ Therefore
\begin{equation}
    \tilde{X}\otimes\tilde{\omega}=\mathcal{L}_{\tilde{X}}g=\mathcal{L}_{fX}g=df\otimes_{s}X^{\natural}+f\mathcal{L}_{X}g = \tilde{X}\otimes (\frac{1}{f}df+\omega).
\end{equation}
It follows that the classes $[\omega],[\tilde{\omega}]$ are the same in $\Lambda^{*}/\lambda^{*}.$ Using the induced metric $g^{\perp}$ we take the dual of $[\omega]$ on $\Lambda/\lambda$ to find a local section in $\Lambda/\lambda.$  

Since the resulting section is independent of which infinitesimal automorphism with closed dual we use, this process extends globally to give a well-defined section $\Phi(g)\in\Omega^{0}(\Lambda/\lambda)$.

By taking the norm with respect to $g^{\perp},$ the map $\Phi:\mathcal{M}\rightarrow \Omega^{0}(\Lambda/\lambda)$ gives rise to a map \begin{equation}\Theta:\mathcal{M}\rightarrow C^{\infty}(M),\end{equation} defined by
\begin{equation}
    \Theta(g)=g^{\perp}(\Phi(g),\Phi(g)),
\end{equation}
for all $g\in \mathcal{M}.$
\section{Degenerate Kundt metrics.}

Recall from \cite{KundtSpacetimes} that a Kundt spacetime $(M,g,\lambda)$ is said to be degenerate if the Riemannian curvature and all its covariant derivative $\nabla^{m}Rm$ are of type $II$ with respect to $\lambda.$ We have the following classification from \cite{NilKilling}
\begin{proposition}\label{DegkundtChar}
Let $(M,g,\lambda)$ be a Kundt space-time and let $X$ be a kundt vector field defined on an open set $U\subset M.$
Then
\begin{enumerate}[i)]
    \item The Riemannian curvature is of type $II$ on $U$ iff. $(\mathcal{L}_{X})^{2}g$ is of boost-order $\leq -2$ w.r.t $\lambda$
    \item $(U,g,\lambda)$ is a degenerate Kundt spacetime iff. $(\mathcal{L}_{X})^{2}g$ is of boost-order $\leq -2$ and $(\mathcal{L}_{X})^3g=0.$
\end{enumerate}
\end{proposition}

Suppose $P\rightarrow M$ is a Kundt structure with corresponding null-quadruple \begin{equation}(\lambda,\Lambda,g^{\perp},f).\end{equation} Let \begin{equation}\mathcal{M}_{Deg}\subset \mathcal{M}\end{equation} be the collection of metrics $g\in \mathcal{M}$ such that the corresponding Kundt spacetime $(M,g,\lambda)$ is degenerate.
As a result of proposition \ref{nil} and corollary \ref{KSKundtCorollary} we can characterize the metrics $g\in\mathcal{M}_{Deg}$ as those for which $(\mathcal{L}_{X})^{2}g$ is of boost-order $\leq -2$ and $(\mathcal{L}_{X})^{3}g=0,$ whenever $X$ is a local infinitesimal automorphism of $P$ belonging to $\lambda.$

We can use this characterization in order to show that $P\rightarrow M$ induces a map
\begin{equation}
    \Psi:\mathcal{M}_{Deg}\rightarrow C^{\infty}(M),
\end{equation}
defined as follows: If $X$ is any local infinitesimal automorphism of $P$ on an open set $U\subset M$ and $g\in \mathcal{M}_{Deg},$ then $(\mathcal{L}_{X})^2g$ is of boost-order $\leq -2$ and therefore there exists a smooth function $H\in C^{\infty}(U)$ such that
\begin{equation}
\mathcal{L}_{X}g=H(X^{\natural}\otimes X^{\natural}).
\end{equation}
If $\tilde{X}$ is another infinitesimal automorphism on $U$, then there exists a smooth function $f$ satisfying $X(f)=0,$ such that $\tilde{X}=fX.$ This gives
\begin{equation}
\begin{gathered}
    \tilde{H}(\tilde{X}^{\natural}\otimes\tilde{X}^{\natural})=(\mathcal{L}_{\tilde{X}})^{2}g\\
    =(\mathcal{L}_{fX})^{2}g=\mathcal{L}_{fX}(df\otimes_{S}X^{\natural}+f\mathcal{L}_Xg)\\
    =f\mathcal{L}_{fX}\mathcal{L}Xg=f[df\otimes (i_{X}\mathcal{L}_{X}g+f(\mathcal{L}_{X})^{2}g]\\
    =f^{2}\mathcal{L}_{X}g=H(\tilde{X}^{\natural}\otimes \tilde{X}^{\natural}),
    \end{gathered}
\end{equation}
where $i_{X}\mathcal{L}_{X}g$ is a contraction, showing that $\tilde{H}=H.$ Since the functions we have constructed are independent of the choice of infinitesimal automorphism, we can use these to construct a global function $\Psi(g)\in C^{\infty}(M)$. 

The metric $g$ also satisfies $(\mathcal{L}_{X})^3g=0,$ and therefore $\mathcal{L}_{X}\Psi(g)=0,$ showing that $\Psi(g)$ is invariant under the local automorphisms of $P.$ 

If $g\in \mathcal{M}_{Deg}$, then we can find coordinates $(u,v,x^{k})$ such that $\frac{\partial}{\partial v}$ is a local infinitesimal automorphism of $P$ and $du=(\frac{\partial}{\partial v})^{\natural}$ and the metric can be expressed by
\begin{equation}\label{DegkundtCoordinates}
    g=2du(dv+Hdu + W_{i}dx^{i}) + \tilde{g}_{ij}(u,x^{k})dx^{i}dx^{j}
\end{equation}
where 
the functions $H,W_{i}$ take the form
\begin{equation}
    H(u,v,x^{k})=v^2H^{(2)}(u,x^k)+vH^{(1)}(u,x^k)+H^{(0)}(u,x^k)
\end{equation}
and
\begin{equation}
W_{i}(u,v,x^{k})=vW^{(1)}_{i}(u,x^{k}) +W^{(0)}_{i}(u,x^{k}).
\end{equation}
In terms of these coordinates we see that $\Psi(g)=2H^{(2)}$, and moreover the section of $\Lambda/\lambda$ given by $\Phi(g)$, which we found in section \ref{KSKundtStructuresSection}, can be expressed in terms of the functions $W_{i}^{(1)}.$

It therefore follows from \cite{KundtSpacetimes} that if $\tilde{g},g\in \mathcal{M}_{Deg}$ satisfy
\begin{equation}
    \Phi(g)=\Phi(\tilde{g})\quad \text{ and } \quad \Psi(g)=\Psi(\tilde{g}),
\end{equation}
Then $g$ and $\tilde{g}$ have identical scalar polynomial curvature invariants. 

Moreover, if $g\in \mathcal{M}_{Deg}$ has constant scalar polynomial curvature invariants across the manifold, then it follows from \cite{CSI4} that $\Psi(g)$ and $\Theta(g)$ differ only by a constant.

\section{Homogeneous GN-structures}
In this section we shall classify left-invariant $GN$ and Kundt structures on homogeneous spaces and study their properties.

We shall consider homogeneous spaces which are quotients $K/H$ of Lie groups $K$ and closed Lie subgroups $H,$ where we assume that the action of $K$ on $K/H$ is effective. By abuse of notation, if $a\in K$ we let $L_{a}$ denote left multiplication of $a$ with elements of both $K$ and $K/H.$  We denote the coset of the identity element $e$ of $K$ by $o$ and refer to it as the origin in $K/H$ and let $\mathfrak{k}$ and $\mathfrak{h}$ be the Lie algebras of $K$ and $H$ respectively. We have a projection $\pi:K\rightarrow K/H$, whose derivative at the identity $(d\pi)_{e}:\mathfrak{k}\rightarrow T_{o}K/H$, induces an isomorphism $\mathfrak{k}/\mathfrak{h}\cong T_{o}K/H,$ which we henceforth use to identify these two vector spaces. Under this identification the vector space quotient $q:\mathfrak{k}\rightarrow \mathfrak{k}/\mathfrak{h}$ and $(d\pi)_{e}$ are the same. If $h\in H,$ then $Ad(h):\mathfrak{k}\rightarrow \mathfrak{k}$ leaves $\mathfrak{h}$ invariant. Therefore we have an induced linear map between vector spaces $Ad(h):\mathfrak{k}/\mathfrak{h}\rightarrow \mathfrak{k}/\mathfrak{h}$. 

For each element $A\in \mathfrak{k},$ the one-parameter subgroup $a_{t}=exp(tA)$, induces a one parameter group of diffeomorphisms on $K/H$ given by left multiplication $L_{a_{t}}$, for $t\in \mathbb{R},$ which defines a vector field which we denote by $A^{*}$ on $K/H.$ Under this association we have 
\begin{equation}
    [A,B]^{*}=-[A^{*},B^{*}],
\end{equation}
for all $A,B\in \mathfrak{k}$. 

On the homogeneous space a $G$-structure $Q\rightarrow K/H$ is said to be $K$-invariant if for each $a\in K,$ the left multiplication $L_{a}:K/H\rightarrow K/H$ is an automorphism of $Q.$ In this case $A^{*}$ is an infinitesimal automorphism of $Q,$ for each $A\in \mathfrak{k},$ since the diffeomorphisms of the flow of $A^{*}$ are given by left multiplication by elements in $K.$

Now let us classify the left invariant $GN$-structures and Kundt structures on $K/H.$ In our classification we shall consider quadruples $(\mathfrak{a},\mathfrak{b},(\cdot,\cdot),\beta)$ where

\begin{enumerate}[i)]
    \item $\mathfrak{h}\subset\mathfrak{a}\subset \mathfrak{b}\subset\mathfrak{k}$ and $\mathfrak{a},\mathfrak{b}$ are $ad(H)$-invariant subspaces such that $\mathfrak{h}$ and $\mathfrak{b}$ are of codimension $1$ in $\mathfrak{a}$ and $\mathfrak{k}$ respectively.   
    \item $(\cdot,\cdot)$ is a positive definite inner product on $\mathfrak{b}/\mathfrak{a}$ which is $ad(H)$ invariant, i.e.  for each $h\in H$ the, linear map $ad(h):\mathfrak{b}/\mathfrak{a}\rightarrow \mathfrak{b}/\mathfrak{a}$ whose existence is ensured by i), is an isometry of $(\cdot,\cdot).$
    \item $\beta:\mathfrak{a}/\mathfrak{h}\otimes \mathfrak{k}/\mathfrak{b}\rightarrow \mathbb{R}$ is a non-zero $ad(H)$-invariant functional.
\end{enumerate}
We shall refer to these as $ad(H)$-invariant quadruples.

\begin{theorem}\label{KSGNHom}
There is a 1:1 correspondence between $K$-invariant $GN$-structures \newline $P\rightarrow K/H$ on $K/H$ and $ad(H)$-invariant quadruples $(\mathfrak{a},\mathfrak{b},(\cdot,\cdot),\beta)$ on $\mathfrak{k}/\mathfrak{h}.$ Under this correspondence the $K$-invariant Kundt Structures are given by the $ad(H)$-invariant quadruples  $(\mathfrak{a},\mathfrak{b},(\cdot,\cdot),\beta)$ such that 
\begin{enumerate}[i)]
\item $\mathfrak{b}$ is a Lie subalgebra of $\mathfrak{k}$,
\item The induced maps $[a,-]\in End(\mathfrak{b}/\mathfrak{a})$ are skew-adjoint with respect to the inner product $(\cdot,\cdot),$ for all $a\in\mathfrak{a}$.
\end{enumerate}
\end{theorem}
\begin{proof}
Suppose that we are given a $K$-invariant $GN$-structure on $K/H$ and let $(\lambda,\Lambda,g^{\perp},f)$ be its associated null-quadruple. From the $K$ invariance it follows that the subspaces $\Lambda_{0}, \lambda_{0}\subset \mathfrak{k}/\mathfrak{h}$
and the metric $g^{\perp}_{0}$ on $\Lambda_{0}/\lambda_{0}$ are $ad(H)$-invariant, and the functional $f_{0}:\lambda_{0}\otimes (\mathfrak{k}/\mathfrak{h})/\Lambda_{0}\rightarrow$ satisfies $f_{0}\circ ad(h)= f_{0},$ for all $h\in H.$ 

Now letting $q:\mathfrak{k}\rightarrow \mathfrak{k}/\mathfrak{h}$ denote the quotient map, we define the following $ad(H)$-invariant subspaces of $\mathfrak{k}$:
\begin{equation}
    \mathfrak{a}:=q^{-1}(\lambda_{0}),\quad\mathfrak{b}:=q^{-1}(\Lambda_{0}),
\end{equation}
then $h$ has co-dimension $1$ in $\mathfrak{a}$ and $\mathfrak{b}$ has co-dimension $1$ in $\mathfrak{k}$ since $\lambda_{0}$ and $\Lambda_{0}$ have dimension $1$ and co-dimension $1$ respectively.

Let $\beta:\mathfrak{a}/\mathfrak{h}\otimes \mathfrak{k}/\mathfrak{b}\rightarrow \mathbb{R}$ be defined through $f_{0}$ by using the identifications $\mathfrak{a}/\mathfrak{h}=\lambda_{0}$ and $\mathfrak{k}/\mathfrak{b}\cong (\mathfrak{k}/\mathfrak{h})/(\mathfrak{b}/\mathfrak{h})=(\mathfrak{k}/\mathfrak{h})/\Lambda_{0}.$ Then by the $ad(H)$-invariance of $f_{0}$ we see that $\beta$ must also be $ad(H)$-invariant.

Lastly we can use the isomorphism $\mathfrak{b}/\mathfrak{a}\cong(\mathfrak{b}/\mathfrak{h})/(\mathfrak{a}/\mathfrak{h})=\Lambda_{0}/\lambda_{0}$, together with $g^{\perp}$ to induce an $ad(H)$-invariant positive definite inner product $(\cdot,\cdot)$ on $\mathfrak{b}/\mathfrak{a}$. 

Thus we have constructed a correspondence from $K$-invariant $GN$-structures to $ad(H)$-invariant quadruples for $\mathfrak{k}/\mathfrak{h}.$ Since any two $K$-invariant $GN$-structures giving the same \newline $ad(H)$-quadruple must have identical fibres over the origin of $K/H$, the $K$-invariance implies that they must coincide everywhere, so the correspondence is injective. Surjectivity follows from the fact that we may induce an $ad(H)$-invariant null-quadruple $(\lambda_{0},\Lambda_{0},g^{\perp},f_{0})$ over the origin of $K/H$ from any $ad(H)$-quadruple $(\mathfrak{a},\mathfrak{b},(\cdot,\cdot),\beta)$ and extend it 
to a $GN$-structure by using the left translations by $K.$
\newline

Let $(\mathfrak{a},\mathfrak{b},(\cdot,\cdot),\beta)$ be an $ad(H)$-invariant quadruple corresponding to a $K$-invariant $GN$-structure $P\overset{\pi}{\rightarrow} K/H$.  

First we show that $P$ is integrable iff. $\mathfrak{b}$ is a Lie subalgebra of $\mathfrak{k}$: Let $\Lambda$ be the codimension $1$ distribution on $K/H$ associated to $P$ and let $\tilde{\Lambda}$ be the distribution on $K$ defined by $\tilde{\Lambda}_{g}=(L_{g})_{*}(\mathfrak{b})$, for all $g\in K.$ From the definition it is clear that $\tilde{\Lambda}$ is integrable iff. $\mathfrak{b}$ is a Lie subalgebra of $\mathfrak{k}$. It suffices therefore to prove that $\Lambda$ is integrable iff. $\tilde{\Lambda}$ is integrable. Let $U\subset K/H$ be a neighborhood such that there exists a local trivalization $\pi^{-1}(U)\overset{\Phi}{\rightarrow} U\times H$ for the principal $H$-bundle $K\overset{\pi}{\rightarrow} K/H$ and a frame $\{X_{1},\dots X_{m}\}$ over $U$ for $\Lambda$. We use $\Phi$ to construct vector fields $\{\tilde{X}_{1},\dots,\tilde{X}_{k}\}$ on $\pi^{-1}(U)$ such that $\pi_{*}(\tilde{X}_{i})=X_{i},$ for $1\leq i\leq k.$ Let $Y_{1},\dots,Y_{m}$ be a basis for $\mathfrak{h}$. Since $\mathfrak{b}$ is $ad(H)$-invariant it follows that $[Y_{i},\tilde{X}_{l}]$ and $[Y_{i},Y_{j}]$ belongs to $\tilde{\Lambda}$ for all $1\leq i,j \leq m$ and $1\leq l\leq k$. Thus $\tilde{\Lambda}$ is integrable on $\pi^{-1}(U)$ iff. $[\tilde{X}_{i},\tilde{X}_{j}]$ belongs to $\tilde{\Lambda}$, for all $1\leq i,j\leq k$. Since $\tilde{X}_{i}$ and $X_{i}$ are $\pi$-related this holds iff. $\Lambda$ is integrable over $U.$ This shows that $\Lambda$ is integrable iff. $\tilde{\Lambda}$ is integrable and hence that $\Lambda$ is integrable iff. $\mathfrak{b}$ is a Lie subalgebra of $\mathfrak{k}.$
 
 Let $\lambda$ be the 1-distribution associated to $P$. Here we show that there are local infinitesimal automorphisms of $P$ pointing along $\lambda$ iff. $[a,-]\in End(\mathfrak{b}/\mathfrak{a})$ is skew-adjoint with respect to $(\cdot,\cdot)$: First we choose a metric $g$ and a vector field $X$ defined on a neighborhood $U$ of the origin in $K/H$ such that $g$ belongs to $P$ and $X$ belongs to $\lambda$. Let $N\in \mathfrak{a}\setminus \mathfrak{h} $ and $W_{i}\in \mathfrak{b}$, for $i=1,2$, then $({N^{*}})_{o}\in \lambda_{o}$ and $({W_{i}}^{*})_{o}\in \Lambda_{o}.$ Since $N^{*}$ is an infinitesimal automorphism of $P$, we see that $\mathcal{L}_{N^{*}}g$ is nilpotent with respect to $X$. Therefore at the origin we have
\begin{equation}\label{Neq}
 \begin{gathered}
 \mathcal{L}_{N^{*}}g(({{W}_{1}}^{*})_{o},({{W}_{2}}^{*})_{o})\\
 =({N}^{*})_{o}g({W_{1}}^{*},{{W}_{2}}^{*}) -g([{N}^{*},{{W}_{1}}^{*}]_{o},({{W}_{2}}^{*})_{o})-g(({{W}_{1}}^{*})_{o},[{N}^{*},{{W}_{2}}^{*}]_{o})\\
 =({N}^{*})_{o}g({W_{1}}^{*},{{W}_{2}}^{*}) +g({[N,W_{1}]^{*}}_{o},({W_{2}}^{*})_{o})+g(({W_{1}}^{*})_{o},{[N,W_{2}]^{*}}_{o})\\
 =({N}^{*})_{o}g({W_{1}}^{*},{{W}_{2}}^{*})+([N,{\overline{W}}_{1}],\overline{W}_{2})+(\overline{W}_{1},[N,\overline{W}_{2}])\\
 =0,
 \end{gathered}
\end{equation}
where $\overline{W}_{i}$ denotes the coset of $W_{i}$ in $\mathfrak{b}/\mathfrak{a}$, for $i=1,2.$
 
 Furthermore since ${{W}_{i}}^{*}$, for $i=1,2,$ are infinitesimal automorphisms of $P$, it follows that
 \begin{equation}\label{eqX}
 \begin{gathered}
     \mathcal{L}_{X}g(({{W}_{1}}^{*})_{o},({{W}_{2}}^{*})_{o}) \\ =X_{o}(g({{W}_{1}}^{*},{{W}_{2}}^{*}))-g([X,{{W}_{1}}^{*}]_{o},({{W}_{2}}^{*})_{o})-g(({{W}_{1}}^{*})_{o},[X,{{W}_{2}}^{*}]_{o})\\
     =X_{o}(g({{W}_{1}}^{*},{{W}_{2}}^{*})),
\end{gathered}     
 \end{equation}
where the last equality holds since $[X,W_{i}^*]\propto X$.

Suppose first that $P$ has local infinitesimal automorphisms belonging to $\lambda$. Then letting $X$ be such we know that $\mathcal{L}_{X}g$ is nilpotent with respect to $X$. Since $(N^{*})_{o}$ and $X_{o}$ are colinear equation \eqref{eqX} implies that
    $({N}^{*})_{o}g({W_{1}}^{*},{{W}_{2}}^{*})=0$ and therefore by \eqref{Neq}, $[N,-]$ is skew-adjoint on $\mathfrak{b}/\mathfrak{a}.$

Conversely, suppose that $[N,-]$ is skew-adjoint on $\mathfrak{b}/\mathfrak{a}$, then by \eqref{Neq} we see that $$({N}^{*})_{o}g({W_{1}}^{*},{{W}_{2}}^{*})=0,$$ but since $({N}^{*})_{o}$ is colinear with $X_{o}$ it follows from \eqref{eqX} that $$\mathcal{L}_{X}g(({{W}_{1}}^{*})_{o},({{W}_{2}}^{*})_{o})=0,$$ and since the $W_{i}$ were arbitrary this in turn implies that $$\mathcal{L}_{X}g(Y,Y^{'})=0,$$
for all $Y,Y^{'}\in \Lambda_{o}.$ Now suppose that $b\in K$, with $\pi(b)\in U.$  Since $L_{b^{-1}}$ preserves $\lambda,$ it follows that there exists some smooth function $f$ such that $(L_{g^{-1}})_{*}(X)=fX.$ Therefore if $Y,Y^{'}\in \Lambda_{o}$, we have
\begin{equation}
\begin{gathered}
    \mathcal{L}_{X}g((L_{b})_{*}Y,(L_{b})_{*}Y^{'})\\
    =(L_{b})^{*}\mathcal{L}_{X}g(Y,Y^{'})=\mathcal{L}_{fX}((L_{b})^{*}g)(Y,Y^{'})=0,
\end{gathered}
\end{equation}
where the last equality holds since $(L_{b})^{*}g$ belongs to $P$. Thus $X$ satisfies \begin{equation}\mathcal{L}_{X}g(\Lambda,\Lambda)=0,\end{equation} on $U.$
    Now if $Z\in \mathfrak{k}\setminus\mathfrak{b}$, then by shrinking $U$ if necessary, we get that $(Z^{*})_{p}\notin \Lambda_{p},$ for all $p\in U,$ and we can find a vector field $X$ on $U$ belonging to $\lambda$ such that $g(X,Z^{*})$ is constant on $U.$ Then
    \begin{equation}
 \mathcal{L}_{{X}}g({X},{Z}^{*})=Xg(X,Z^{*})-g([{X},{X}],{Z}^{*})-g({X},[{X},{Z}^{*}])=-g({X},[X,Z^{*}])=0,    
 \end{equation}
 Thus $X$ is an infinitesimal automorphism for $P$ on $U.$ If $p\in K/H$, then we may find some $a\in K$ such that $\pi(a)=p.$ Since 
 \begin{equation}
     \mathcal{L}_{(L_{a})_{*}X}((L_{a^{-1}})^{*}g)=(L_{a^{-1}})^{*}\mathcal{L}_{X}g,
 \end{equation}
 it follows that $(L_{a})_{*}(X)$ is an infinitesimal automorphism of $P$ defined on $L_{a}(U)$. Hence if $[N,-]\in End(\mathfrak{b}/\mathfrak{a})$, for all $N\in \mathfrak{a}$, then there are locally defined infinitesimal automorpisms of $P$ about each point of $K/H.$ This concludes the proof.

\end{proof}


Given a Lie group $G$ with Lie algebra $\mathfrak{g}$ we now consider left-invariant $GN$-structures $P\rightarrow G$. By theorem \ref{KSGNHom} these are in $1:1$ correspondece with quadruples $(\mathfrak{a},\mathfrak{b},(\cdot,\cdot),\beta),$ where $\mathfrak{a}$ and $\mathfrak{b}$ are of dimension one and codimension one respectively and satisfy $\mathfrak{a}\subset \mathfrak{b}\subset \mathfrak{g}.$  

\begin{example}
Every 3-dimensional Lie group whose Lie algebra $\mathfrak{g}$ is not isomorphic to $\mathfrak{su}(2)$ admits a left-invariant Kundt structure.

This can be seen as follows: If $\mathfrak{g}$ is not isomorphic to $\mathfrak{su}(2)$, then it has a codimension 1 subalgebra $\mathfrak{b}$. If $\mathfrak{b}$ is abelian, then choose any one-dimensional subspace $\mathfrak{a}\in \mathfrak{b}$. If $\mathfrak{b}$ is non-abelian, then set $\mathfrak{a}=[\mathfrak{b},\mathfrak{b}]$ which is automatically one-dimensional. In either case the induced map $[A,-]:\mathfrak{b}/\mathfrak{a}\rightarrow \mathfrak{b}/\mathfrak{a}$ is zero, for all $A\in \mathfrak{a}.$ Therefore if $(\cdot,\cdot)$ is any quadratic form on $\mathfrak{b}/\mathfrak{a}$ and $\beta:\mathfrak{a}\otimes \mathfrak{g}/\mathfrak{b}\rightarrow \mathbb{R}$ is any non-zero linear function, the quadruple $(\mathfrak{a},\mathfrak{b},(\cdot,\cdot),\beta)$ corresponds to a left-invariant Kundt structure.

$\mathfrak{su}(2)$ does not contain a codimension one subalgebra, and therefore does not support a quaduple corresponding to a Kundt structure.
\end{example}

The next result suggests that in arbitrary dimension left-invariant Kundt structures need not be as ubiquitous.

\begin{proposition}
Suppose that $N$ is a nilpotent Lie group with Lie algebra $\mathfrak{n}.$ If $(\mathfrak{a},\mathfrak{b},(\cdot,\cdot),\beta)$ is a quaduple on $\mathfrak{n}$ corresponding to a left-invariant Kundt structure $P\rightarrow M$, then
\begin{equation}
    [\mathfrak{a},\mathfrak{b}]\subset \mathfrak{a}\quad \text{ and } \quad [\mathfrak{a},\mathfrak{n}]\subset \mathfrak{b}.
\end{equation}
Consequently any left-invariant vector field in $\mathfrak{a}$ is an infinitesimal automorphism of $P.$
\end{proposition}
\begin{proof}
If $A\in \mathfrak{a}$, then $ad(A)$ is nilpotent, and therefore the induced map $ad(A):\mathfrak{b}/\mathfrak{a}\rightarrow \mathfrak{b}/\mathfrak{a}$ is also nilpotent. Since this map is assumed to skew-adjoint it follows that it must be identically zero. This shows that $[A,B]\subset \mathfrak{a}$, for all $B\in \mathfrak{b},$ from which we see that $[\mathfrak{a},\mathfrak{b}]\subset \mathfrak{a}.$

Since $\mathfrak{b}$ is a subalgebra, it follows again from nilpotency by a standard argument that $[\mathfrak{a},\mathfrak{n}]\subset \mathfrak{b}.$

Now let $X$ be a left-invariant vector field on $N,$ such that at the identify $X_e\in \mathfrak{a}$. Then left invariance implies that $X$ belongs to the null-distribution $\lambda$ associated to the Kundt structure. If $g$ is a metric belonging to $P,$ then left-invariance of $(\cdot,\cdot)$ and $\beta$ implies that if $W,W^{\prime}\in \mathfrak{b}$ and $Z\in \mathfrak{n}$, then $g(W,W^{'})$ and $g(X,Z)$ are contant on $N,$ and therefore
\begin{equation}
    \mathcal{L}_{X}g(W,W^{\prime})=-g([X,W],W^\prime)-g(W,[X,W^{\prime}])=0,
\end{equation}
and 
\begin{equation}
    \mathcal{L}_{X}g(X,Z)=-g([X,X],Z)-g(X,[X,Z])=0,
\end{equation}
from which we see that $X$ is nil-Killing w.r.t. $\lambda.$ Hence proposition \ref{nil} implies that $X$ is an infinitesimal automorphism of $P.$
\end{proof}

\section{Conclusion}
In this paper we have identified the $G$-structures to which Kundt and nil-Killing vector fields are infinitesimal automorphisms. The properties of these $G$-structures are studied showing that they give rise to an intrinsic boost-order classification of tensors and the ability to perform full traces of even ranked tensors of type $II$.

Along the way we have shown that the existence of a Kundt vector field can be characterized in terms of the $G$-structure having a torsion-free connection. 

Lastly, our desire to represent Kundt-CSI spacetimes in terms of nil-Killing vector fields has lead us to characterize left-invariant Kundt structures on homogeneous spaces. Here we observed some special restrictions such a structure places upon nilpotent Lie groups.

\section*{Acknowledgements}
I would like to thank Boris Kruglikov, Lode Wylleman and my PhD advisor Sigbjørn Hervik for helpful discussions concerning this project.

\medskip

\bibliography{KundtStructures}
\bibliographystyle{plain}
\end{document}